\documentclass[11pt,reqno]{amsart}
\usepackage[T1]{fontenc}
\usepackage{lmodern,microtype}
\usepackage{amsmath,amssymb,mathtools,booktabs}
\usepackage[margin=1in]{geometry}
\usepackage[hidelinks]{hyperref}
\usepackage[nameinlink,noabbrev]{cleveref}
\numberwithin{equation}{section}
\newtheorem{theorem}{Theorem}[section]
\newtheorem{lemma}[theorem]{Lemma}

\theoremstyle{remark}\newtheorem{remark}[theorem]{Remark}
\newcommand{\Vol}{\operatorname{Vol}}

\newcommand{\R}{\mathbb R}
\newcommand{\dd}{\,\mathrm d}
\allowdisplaybreaks
\hypersetup{
  pdftitle={A 0.41275 lower bound for the three-dimensional Blaschke--Lebesgue problem},
  pdfauthor={Yunjiang Jiang},
  pdfsubject={Convex bodies of constant width in Euclidean three-space},
  pdfkeywords={Blaschke--Lebesgue problem, constant width, Meissner conjecture, contact tetrahedron, spindle convexity}
}
\title[A $0.41275$ lower bound for the Blaschke--Lebesgue problem]
{A $0.41275$ lower bound for the three-dimensional Blaschke--Lebesgue problem}
\author{Yunjiang Jiang}
\date{}
\begin{document}
\begin{abstract}
We prove that every convex body $K\subset\R^3$ of constant width $w>0$ satisfies
\[
   \Vol(K)\ge \frac{1651}{4000}w^3=0.41275\,w^3.
\]
This improves the certified lower bound $0.411040473721188\ldots\,w^3$ announced by Hyra and the earlier spectral bound $4\pi w^3/33$ of Nishioka. The main new ingredient is a contact-tetrahedron estimate near equality in Jung's circumradius inequality. Four balanced contacts of a minimum circumscribed ball form a near-regular tetrahedron; constant width then forces, beyond the tetrahedron itself, spherical regions over its faces, spindle sectors along its edges, and additional face-edge sectors. These regions can be selected with pairwise disjoint interiors and quantified by the total squared-edge deficit. A sharp convexity argument relates that deficit to the circumradius, and a symmetrization argument controls the four face contributions with only quadratic loss. This gives the analytic bound $\Vol(K)>0.41275318$ near the Jung endpoint.

For the complementary circumradius range we use Hyra's radial transport/tangent-majorant mechanism together with closed-form contact and antipodal profiles. The antipodal contribution has an elementary primitive. For the contact contribution, a second-order concavity estimate converts the angular Jacobian into an explicit rational lower density; that density is monotone in the circumradius. The whole intermediate range therefore reduces to a single endpoint calculation, giving $\Vol(K)>0.41275168$. For smaller circumradius the tangent-majorant inequality with its nonnegative gap discarded already gives $\Vol(K)>0.41308109$. The proof is entirely analytic and requires no computer-assisted certificate.
\end{abstract}
\maketitle

\section{Introduction}
A convex body $K\subset\R^3$ has \emph{constant width} $w>0$ if the distance between its two supporting planes orthogonal to any unit vector is $w$. Write
\[
 m_3(w)=\inf\{\Vol(K):K\subset\R^3\text{ is a convex body of constant width }w\}.
\]
Scaling gives $m_3(w)=w^3m_3(1)$. The three-dimensional Blaschke--Lebesgue problem asks for $m_3(1)$ and for the bodies that attain it.

\subsection{Historical context}
The problem belongs to one of the oldest threads in convex-geometric shape optimization. Meissner studied three-dimensional sets of constant width in 1911 \cite{Meissner1911}, and Meissner and Schilling described physical models of the now classical constant-width modifications of the Reuleaux tetrahedron in 1912 \cite{MeissnerSchilling1912}; see also the historical account of Kawohl and Weber \cite{KawohlWeber2011}. In the plane, Lebesgue and Blaschke independently proved in 1914--1915 that the Reuleaux triangle minimizes area among bodies of prescribed constant width \cite{Lebesgue1914,Blaschke1915}. Bonnesen and Fenchel subsequently conjectured in their 1934 monograph that the two noncongruent Meissner tetrahedra minimize volume in dimension three \cite{BonnesenFenchel1934}. Thus the standard minimizer conjecture has remained open for more than ninety years, while its geometric roots go back more than a century. It is appropriately described as one of the longest-standing open extremal problems in convex geometry.

The common volume of the two Meissner tetrahedra of width $w$ is
\begin{equation}\label{eq:meissner-volume}
 V_{\mathrm M}(w)=\frac{\pi}{12}
 \left(8-3\sqrt3\arccos\frac13\right)w^3
 =0.419860045965080\ldots\,w^3;
\end{equation}
see \cite{KawohlWeber2011,Bogosel2026}. Several partial structural results are known. Campi, Colesanti, and Gronchi solved the restricted minimum-volume problem for bodies of revolution \cite{CampiColesantiGronchi1996}. Bayen, Lachand-Robert, and Oudet developed an analytic parametrization and weak optimality conditions for three-dimensional bodies of constant width \cite{BayenLachandRobertOudet2007}. Anciaux and Guilfoyle proved that the smooth pieces of a minimizer must have constant smaller principal curvature, hence are spherical caps or pieces of canal surfaces \cite{AnciauxGuilfoyle2011}. Montejano and Rold\'an-Pensado developed the modern theory of Meissner polyhedra \cite{MontejanoRoldan2017}; Hynd proved their density among constant-width bodies and later their extremality under Minkowski addition \cite{Hynd2024,Hynd2026}; Bogosel obtained explicit volume formulas and finite-dimensional reductions \cite{Bogosel2026}; and Arman, Bondarenko, Prymak, and Radchenko constructed further small-volume bodies with tetrahedral symmetry \cite{ArmanEtAl2025}.

The classical universal lower bound of Chakerian is
\[
 \Vol(K)\ge \frac{\pi}{3}(3\sqrt6-7)w^3
 =0.364916122595\ldots\,w^3
\]
\cite{Chakerian1966}. Nishioka recently improved this to
\begin{equation}\label{eq:nishioka-intro}
 \Vol(K)\ge\frac{4\pi}{33}w^3
 =0.380799109526\ldots\,w^3
\end{equation}
by a support-function and spherical-harmonic argument \cite{Nishioka2026}. Very recently, Hyra introduced a different geometric method based on minimum-circumsphere contacts, an opposite-point identity, radial transport, and a tangent majorant. The method gives the analytic lower bound $0.404120778796972\ldots\,w^3$ and, after a finite exact-rational cap certificate, the stronger bound
\begin{equation}\label{eq:hyra-intro}
 \Vol(K)>0.411040473721188\,w^3
\end{equation}
\cite{Hyra2026}.

Our main theorem is the following.
\begin{theorem}[Main theorem]\label{thm:main}
Every convex body $K\subset\R^3$ of constant width $w>0$ satisfies
\begin{equation}\label{eq:main}
 \boxed{\Vol(K)\ge\frac{1651}{4000}w^3=0.41275\,w^3.}
\end{equation}
\end{theorem}
Together with \eqref{eq:meissner-volume}, this gives
\[
 0.41275\le \frac{m_3(w)}{w^3}
 \le0.419860045965080\ldots.
\]
The theorem does not identify a minimizer and does not prove the Meissner conjecture.

\subsection{The proof at a glance}
Normalize $w=1$ and place the center of a minimum circumscribed ball at the origin. Let $R$ be its radius. Jung's inequality gives
\[
 \frac12\le R\le\sqrt{\frac38}.
\]
Put
\[
 R_1=\frac{151}{250}=0.604,
 \qquad
 R_0=\frac{152597}{250000}=0.610388.
\]
The proof has three analytic regimes.

\paragraph{Small circumradius: $R\le R_1$.}
Use the radial transport inequality with the fixed choices
\[
 m=1-R_1^2,\qquad s=\frac{231}{500}.
\]
Since $m\le1-R^2$, the inequality is valid throughout this range. Discarding its nonnegative tangent gap gives directly
\[
 \Vol(K)>0.4130810985.
\]

\paragraph{Intermediate circumradius: $R_1\le R\le R_0$.}
Set
\[
 c=\frac{571}{5000},\qquad m=1-R^2,\qquad s=R-c.
\]
Around each of the four balanced circumsphere contacts there is a closed-form radial lower profile, and around each antipode there is a closed-form radial upper profile. Contact balance makes the four same-sign caps disjoint. The antipodal cap ends at the level $\rho=s$ before the disjointness threshold is reached, and its entire tangent-gap contribution has an elementary rational primitive.

The positive contact cap is the only delicate term. Writing its profile in terms of the output radius $v$, the angular Jacobian contains a factor $B^{-1}$ with $B=\sqrt{R^2-A^2}$. Concavity gives the second-order upper estimate
\[
 B\le \gamma-\frac{\alpha}{\gamma}x-
        \frac{R^2}{2\gamma^3}x^2,
 \qquad x=A-\alpha,
\]
which turns the resulting cap contribution into an explicit rational lower density $\Phi_R(z)$. We prove pointwise that $\Phi_R(z)$ decreases with $R$. The baseline plus antipodal contribution also decreases with $R$. Consequently the whole interval is controlled by the single endpoint $R=R_0$. A short Taylor-coefficient estimate at that rational endpoint gives
\[
 \Vol(K)>0.4127516885.
\]
No mesh or finite numerical certificate is used.

\paragraph{Near-Jung circumradius: $R_0\le R\le\sqrt{3/8}$.}
Here four balanced contacts form a quantitatively near-regular tetrahedron $T$. If
\[
 D=\sum_{i<j}\bigl(1-|x_i-x_j|^2\bigr)
\]
is its total squared-edge deficit and $d=3/8-R^2$, a matrix-convexity argument gives the sharp comparison
\begin{equation}\label{eq:intro-sharpD}
 D\le\frac{16d}{1-8d}.
\end{equation}
Constant width forces considerably more volume than $T$ itself: spherical regions above the four faces, spindle sectors along the six edges, and additional spherical sectors between the face regions and the edge spindles. They can be chosen with disjoint interiors. Symmetrizing the four face contributions cancels their first-order anisotropy, leaving only quadratic loss. The resulting explicit function $B(D)$ from \cref{thm:geometric} gives
\[
 \Vol(K)>0.41275318
\]
at $R=R_0$, and only improves as $R$ approaches the Jung endpoint.

\subsection{Exact constants and computational scope}
The argument is computer-free. Every terminating decimal appearing in an inequality denotes the corresponding exact rational number. The few transcendental comparisons used at the end are derived from elementary rational bounds for $\pi$, $\sqrt3$, $\sqrt{2/3}$, and $\arcsin(1/3)$; for completeness these bounds and their short series proofs are recorded in \cref{lem:constants}. Thus no ancillary verifier, floating-point decision, optimization output, or hidden numerical certificate is part of the proof.

\subsection{Organization}
\Cref{sec:constant-width} develops constant-width identities, circumsphere contacts, and the sharp edge-deficit comparison. The next two sections quantify the volume forced by the contact tetrahedron: first the tetrahedron and face regions, then the edge spindles and the additional face-edge sectors. The final section develops the analytic radial estimate, records the elementary rational bounds for the few transcendental constants, and assembles the three circumradius regimes.

\section{The near-Jung geometric estimate}\label{sec:geometric-overview}
We first state the quantitative estimate that will be proved in \cref{sec:constant-width} and the following two sections. Put
\begin{align*}
 I_*&=\frac{151889}{2500000}=0.0607556,&
 b_0&=\frac{\sqrt3}{2},& k_0&=\frac{\sqrt3}{6},\\
 \phi_0&=\arcsin(1/3),& \beta_0&=\frac\pi2+\phi_0,&
 J_0&=\frac{11}{12}-\frac{\pi\sqrt3}{6},\\
 j_0&=\frac{\pi}{12\sqrt3}-\frac18,&
 E_0&=J_0(b_0\phi_0-k_0).
\end{align*}
For $0\le D\le0.06$, define
\begin{equation}\label{eq:B}
\begin{split}
 B(D)={}&\frac{\sqrt{(1-D)(2+D)}}{12}
 +4I_*\left(1-\frac D6\right)^{3/2}+3\beta_0J_0+12E_0\\
 &-\left(\frac{\beta_0j_0}{2}+0.00069\right)D-0.081D^2.
\end{split}
\end{equation}
All terminating decimals in the proof denote the corresponding exact rational numbers.

\begin{theorem}[Near-Jung geometric estimate]\label{thm:geometric}
Suppose a width-one body $K$ has circumradius $R\ge0.609$, and let $T$ be a balanced contact tetrahedron. If the six squared-edge deficits of $T$ are nonnegative and have sum $D\le0.06$, then
\[
 \Vol(K)\ge B(D).
\]
\end{theorem}

\section{Constant width and a sharper edge-deficit estimate}\label{sec:constant-width}
A width-one body satisfies
\begin{equation}\label{eq:balls}
 K=\bigcap_{y\in K}B(y,1).
\end{equation}
Indeed its diameter is one. Conversely, if $z\notin K$, choose a unit separating direction $u$ and a point $y$ on the opposite supporting plane. The width identity gives $(z-y)\cdot u>1$, so $z\notin B(y,1)$. If $x\in\partial K$ has outer unit normal $n$, then
\begin{equation}\label{eq:opposite}
 x-n\in K.
\end{equation}
To see this, choose $y$ on the opposite supporting plane; $(x-y)\cdot n=1$ and $|x-y|\le1$ force $y=x-n$.

A minimum enclosing ball has balanced contacts $x_i=Ru_i$:
\[
 \sum_{i=1}^n\lambda_i x_i=0,\qquad \lambda_i>0,\quad
 \sum_i\lambda_i=1,\quad 2\le n\le4.
\]
The origin belongs to the convex hull of the contact set: otherwise a strict separating hyperplane gives a direction in which a sufficiently small translation of the center decreases the distance to every contact, while points away from the contact set retain positive radial slack. Caratheodory's theorem then supplies at most four contacts with positive balancing weights. Set $\eta=1/(2R^2)-1$. Since $u_i\cdot u_j\ge-\eta$, taking a scalar product with $u_i$ gives
\begin{equation}\label{eq:lambda}
 \lambda_i\le\frac{\eta}{1+\eta}=1-2R^2.
\end{equation}
Thus $1/2\le R\le\sqrt{3/8}$, and $R>1/\sqrt3$ requires four contacts. This also proves the needed case of Jung's inequality \cite{Jung}.

Write
\[
 \delta_{ij}=1-|x_i-x_j|^2,\quad D=\sum_{i<j}\delta_{ij},
 \quad d=\frac38-R^2,\quad b=\frac14\sum_i x_i.
\]
The elementary identity
\begin{equation}\label{eq:Didentity}
 D=16d+16|b|^2
\end{equation}
follows from $\sum_{i<j}|x_i-x_j|^2=16(R^2-|b|^2)$.
The weight bound implies $|b|\le6d$: the positive part of the four numbers $\lambda_i-1/4$ has sum at most $6d$, and a zero-sum combination of points of diameter one has norm at most this positive sum. Consequently $D\le16d+576d^2<1$ when $R\ge0.6$.

\begin{lemma}[Sharp deficit comparison]\label{lem:sharpD}
For four points on a common sphere, with $0\le\delta_{ij}$ and $D<1$,
\[
 |b|^2\le\frac{D^2}{16(2+D)}.
\]
For a balanced contact tetrahedron with $R\ge0.6$, this gives
\begin{equation}\label{eq:sharpD}
 \boxed{D\le\frac{16d}{1-8d}.}
\end{equation}
Equality in the first inequality is attained on the one-shortened-edge vertices of the deficit simplex.
\end{lemma}
\begin{proof}
Let $P=I-\frac14\mathbf1\mathbf1^t$, and let $\Delta$ have off-diagonal entries $\delta_{ij}$ and zero diagonal. On $H=\mathbf1^\perp$, the centered Gram matrix is
\[
 G=\frac12(P+P\Delta P).
\]
Since $\|\Delta\|_{\rm op}\le D<1$, $G$ is positive definite on $H$. Put $r_i=\sum_j\delta_{ij}$ and $z=(r-(D/2)\mathbf1)/8$. The equal-radius identities give
\[
 b\cdot(x_i-b)=z_i,\qquad |b|^2=z^t(G|_H)^{-1}z.
\]
For completeness, the diagonal entries of $G$ are $3/8-r_i/4+D/16$, whereas their average is $3/8-D/16$; subtracting the equal-radius identities proves the first formula.

For fixed $D$, the function in the last display is convex in the six deficits, because
\[
 z^tG^{-1}z=\sup_{v\in H}\{2z\cdot v-v^tGv\}
\]
is a supremum of affine functions. The deficit simplex has vertices with one deficit equal to $D$ and all others zero. At such a vertex, $|z|^2=D^2/64$ and $Gz=((2+D)/4)z$. Hence every vertex value is $D^2/[16(2+D)]$. Convexity proves the bound. Substitution in \eqref{eq:Didentity} gives $2D/(2+D)\le16d$, equivalent to \eqref{eq:sharpD}.
\end{proof}

\section{Volume of the tetrahedron and strengthened face regions}
Throughout the geometric estimates, assume $D\le0.06$. In particular all faces are acute: each squared side is at least $0.94$, and the sum of any two exceeds the third. Every face has squared circumradius $q\le1/3$: its largest angle lies in $[\pi/3,\pi/2)$, its longest side is at most one, and the sine law gives this bound. The balanced origin is strictly inside the nondegenerate tetrahedron, so every outward face plane has positive distance from the origin.

\begin{lemma}\label{lem:tet}
The tetrahedron satisfies
\[
 \Vol(T)\ge\frac{\sqrt{(1-D)(2+D)}}{12}.
\]
\end{lemma}
\begin{proof}
A Gram matrix of three incident edge vectors has diagonal entries $a,b,c$ and off-diagonal entries $(a+b-d)/2$, $(a+c-e)/2$, $(b+c-f)/2$, where $a,\ldots,f$ are the six squared lengths. This matrix is affine in the deficits. At a vertex of their simplex, direct calculation gives determinant
\[
 \frac12-\frac D4-\frac{D^2}{4}=\frac{(1-D)(2+D)}4.
\]
The six vertex matrices are positive definite for the present range of $D$. Concavity of $\log\det$ therefore bounds the determinant of every convex combination below by this common value. This concavity follows, for example, by differentiating twice along a line of positive matrices: the second derivative is $-\operatorname{tr}(G^{-1}G'G^{-1}G')\le0$. Finally $\Vol(T)=\sqrt{\det G}/6$.
\end{proof}

\subsection{Using the enclosing ball}
Let $F$ be a face with outward unit normal $n$ and plane $p\cdot n=h>0$. Its planar circumcenter is $hn$, and its squared circumradius is $q=R^2-h^2$. For $p\in F$, put
\[
 W(p)=R^2-|p|^2=q-|p-hn|^2,\qquad \kappa=R+h.
\]
\begin{lemma}[Circumradius-aware face region]\label{lem:faceinclusion}
The following entire region lies in $K$:
\begin{equation}\label{eq:newface}
 \left\{p+tn:p\in F,\quad
 0\le t\le\sqrt{\kappa^2+W(p)}-\kappa\right\}.
\end{equation}
\end{lemma}
\begin{proof}
Express $p=\sum_i\mu_i x_i$ using the face vertices. For $y\in K$, the inequalities $|y-x_i|\le1$ imply
\[
 |y|^2-2p\cdot y+R^2\le1.
\]
For $x=p+tn$, use $n\cdot y\ge-R$ to get
\[
 |x-y|^2\le1-R^2+|p|^2+t^2+2t(R+h)\le1
\]
whenever $t$ is in \eqref{eq:newface}. Equation \eqref{eq:balls} proves the inclusion.
\end{proof}
These face regions project orthogonally into their respective faces. We next estimate their total volume, denoted $\sum_F I_F^{\rm new}$.

\subsection{Symmetrizing the faces}
Let $a_1,a_2,a_3$ be the squared lengths of a face, $\tau=3-\sum_i a_i$, $s=1-\tau/3$, and $v_i=a_i-s$. Write $A$ for its area and $A_0=\sqrt3/4$. Heron's formula and the circumradius formula give
\[
 16A^2=3s^2-2\sum_i v_i^2,\qquad q\ge\frac{s}{3},
 \qquad A_0s-0.1\tau^2\le A\le A_0s.
\]
To justify $q\ge s/3$, put $P=a_1+a_2+a_3$ and $H=2\sum_{i<j}a_ia_j-\sum_i a_i^2$; the inequality $9a_1a_2a_3\ge PH$ is Schur's degree-three inequality. For $a\ge b\ge c$, its difference is
$(a-b)^2(a+b-c)+c(a-c)(b-c)\ge0$.
The area bound follows from $\sum v_i^2\le2\tau^2/3$ and $\sqrt H>5/3$ when $\tau\le0.06$.

For a uniform barycentric point, let
\[
 W_0=\mu_1\mu_2+\mu_2\mu_3+\mu_3\mu_1,
 \qquad W=\sum_{i=1}^3a_i\mu_j\mu_k
\]
(the index $i$ denotes the side opposite $\mu_i$). Define
\[
 \kappa_s=R+\sqrt{R^2-s/3},\qquad
 f(\kappa,W)=\sqrt{\kappa^2+W}-\kappa.
\]
Because $q\ge s/3$, the actual $\kappa$ is at most $\kappa_s$, so the new face volume is at least $A\mathbb E f(\kappa_s,W)$.

For $R\ge0.609$, all these $\kappa_s$ exceed $0.8$. Since
$f_{WW}=-1/[4(\kappa^2+W)^{3/2}]$, Taylor's theorem gives
\[
 \mathbb E f(\kappa_s,W)
 \ge\mathbb E f(\kappa_s,sW_0)-\frac{\tau^2}{2160\kappa_s^3}.
\]
Indeed the linear term vanishes by permutation symmetry and $\sum v_i=0$, while the barycentric moments imply
\[
 \mathbb E(W-sW_0)^2=\frac{\sum_i v_i^2}{180}\le\frac{\tau^2}{270}.
\]
Moreover $\mathbb E f(\kappa_s,sW_0)\le\sqrt{0.8^2+1/4}-0.8<0.15$ by concavity. It follows that
\begin{equation}\label{eq:faceanisotropy}
 I_F^{\rm new}\ge A_0s\,\mathbb E f(\kappa_s,sW_0)-0.0154\tau^2.
\end{equation}
Here $A_0/(2160(0.8)^3)<0.0004$, explaining the displayed error constant.

For fixed $R$ and $w\in[0,1/3]$, the function
\[
 s\longmapsto s\,f\bigl(R+\sqrt{R^2-s/3},sw\bigr)
\]
is convex. To check this directly, abbreviate its inner $f$ by $g(s)$ and $\kappa_s$ by $\kappa$. The identity $g^2+2\kappa g=sw$ gives
\[
 g'=\frac{w/2-\kappa'g}{g+\kappa},\quad 0\le g'\le-\kappa',\qquad
 (g+\kappa)g''=-(g')^2-2\kappa'g'-\kappa''g\ge0.
\]
The middle inequality uses $w\le\kappa/(3\sqrt{R^2-s/3})$; its right side is at least $2/3$. Thus $(sg)''\ge0$.

Set $\bar s=1-D/6$. Each edge belongs to two faces, so $\sum_F\tau_F=2D$ and $\sum_F\tau_F^2\le2D^2$. Jensen's inequality and \eqref{eq:faceanisotropy} give
\begin{equation}\label{eq:facesfirst}
 \sum_F I_F^{\rm new}\ge
 4A_0\bar s\,\mathbb E f(\kappa_{\bar s},\bar s W_0)-0.031D^2.
\end{equation}

\subsection{Comparison with the regular face}
Let $I_{\rm reg}=A_0\mathbb E[\sqrt{2/3+W_0}-\sqrt{2/3}]$.
By Lemma~\ref{lem:sharpD} and \eqref{eq:Didentity},
\[
 R^2=\frac38\bar s+|b|^2,\qquad
 |b|^2\le\frac{D^2}{16(2+D)}\le\frac{D^2}{32}.
\]
Writing $h_0=\sqrt{2/3}$, rationalization gives
\[
 0\le\kappa_{\bar s}-h_0\sqrt{\bar s}
 \le\frac{4h_0}{\sqrt{\bar s}}|b|^2.
\]
Since $-f_\kappa=1-\kappa/\sqrt{\kappa^2+W}\le W/(2\kappa^2)$, its expectation for $W=\bar sW_0$ and $\kappa\ge h_0\sqrt{\bar s}$ is at most $3/16$. Therefore \eqref{eq:facesfirst} implies
\begin{equation}\label{eq:facesfinal}
 \boxed{\sum_F I_F^{\rm new}\ge4I_{\rm reg}(1-D/6)^{3/2}-0.065D^2.}
\end{equation}
The extra error is at most $(3\sqrt2/128)D^2$, and $0.031+3\sqrt2/128<0.065$.

The small numerical input $I_{\rm reg}>I_*$ has a short finite proof. Put $Q=1/3-W_0$. Its first five moments are
\[
 \mathbb E Q=\frac1{12},\quad \mathbb E Q^2=\frac1{90},\quad
 \mathbb E Q^3=\frac{29}{15120},\quad
 \mathbb E Q^4=\frac{11}{28350},\quad
 \mathbb E Q^5=\frac{13}{149688}.
\]
These follow by expanding $Q^j$ and using
\[
 \mathbb E\mu_1^a\mu_2^b\mu_3^c=\frac{2a!b!c!}{(a+b+c+2)!}.
\]
If $c_j=\binom{2j}{j}/[4^j(2j-1)]$, the binomial expansion and $0\le Q\le1/3$ give
\[
 \sqrt{1-Q}\ge1-\sum_{j=1}^4c_jQ^j-\frac32c_5Q^5.
\]
Consequently
\[
 I_{\rm reg}\ge\frac{\sqrt3}{4}
 \left(1-\sqrt{\frac23}-\sum_{j=1}^4c_j\mathbb E Q^j
       -\frac32c_5\mathbb E Q^5\right)
 >I_*.
\]
The rational bounds $1732050807/10^9<\sqrt3$ and $\sqrt{2/3}<816496581/10^9$, verified by squaring, make this last comparison an exact rational inequality.

\section{Edge spindles and the previously uncounted sectors}
\subsection{Spindles and projection disjointness}
For an edge of length $\ell$, take its midpoint as origin and endpoints $\pm(\ell/2)e$. Set
\[
 b_\ell=\sqrt{1-\ell^2/4},\quad
 t_0(z)=\sqrt{1-z^2}-b_\ell\quad (|z|\le\ell/2),
\]
\[
 J(\ell)=\int_{-\ell/2}^{\ell/2}t_0(z)^2\dd z
 =\ell-\frac{\ell^3}{12}-2b_\ell\arcsin(\ell/2).
\]
Every unit ball containing the endpoints contains the spindle $ze+tv$, where $v\perp e$, $|v|=1$, and $0\le t\le t_0(z)$. Indeed, if its center is $y=qe+y_\perp$, endpoint containment gives $|y|^2+\ell|q|\le b_\ell^2$, and
\[
 |ze+tv-y|^2\le z^2+(t+b_\ell)^2\le1.
\]
Retain directions $v$ in the exterior normal cone of $T$ at that edge. Its angular measure is the exterior dihedral angle $\beta_e$, so this region has volume $\beta_eJ(\ell_e)/2$.

Every point in an edge region has its nearest point on $T$ on that edge. The new face regions have nearest point in their respective faces. Unique nearest-point projection proves disjointness, except on null boundaries, both between these regions and with $T$ itself.

\subsection{Quantitative angle control}
We claim, for $D\le0.06$,
\begin{equation}\label{eq:angles}
 \beta_e\le\beta_0+1.1D,\qquad
 \sum_e\beta_e\ge6\beta_0-0.3D^2,\qquad \beta_e>\pi/2.
\end{equation}
Here are explicit derivative bounds proving the claim. Interpolate squared lengths as $1-t\delta_e$. For an edge of squared length $a$ put
\[
 u=(a+b-d)/2,\quad v=(a+c-e)/2,\quad w=(b+c-f)/2,
 \quad N=aw-uv,\quad P=ab-u^2,\quad Q=ac-v^2.
\]
Then $\cos\beta=-N/\sqrt{PQ}$. If the six deficits are $A,B,C,D_1,E,F_1$, expansion gives
\begin{align*}
 N={}&\frac14-\frac{B+C+D_1+E}{4}+\frac{F_1}{2}
 -\frac{A^2}{4}+\frac{A(B+C+D_1+E)}4-\frac{AF_1}{2}\\
 &-\frac{BC}{4}+\frac{BE}{4}+\frac{CD_1}{4}-\frac{D_1E}{4},\\
 P={}&\frac34-\frac{A+B+D_1}{2}
 -\frac{A^2+B^2+D_1^2}{4}+\frac{AB+AD_1+BD_1}{2}.
\end{align*}
On the nonnegative deficit simplex of total at most $0.06$, these formulas give
\[
 0<N\le0.2818,\quad P,Q\ge0.7191,\quad
 |N'|\le0.5D,\quad |P'|,|Q'|\le0.53D.
\]
For example the absolute partial derivatives of $N$ are bounded by $0.5$; those of $P,Q$ by $0.53$. Thus $0<N/\sqrt{PQ}<0.392$, $\sin\beta>0.919$, and
\[
 |\beta'|\le\frac{0.5+0.392(0.53)}{0.7191(0.919)}D<1.1D.
\]
The Euclidean Schlaefli identity $\sum_e\ell_e\beta'_e=0$ now gives
\[
 \left|\frac\dd{\dd t}\sum_e\beta_e\right|
 \le \frac{25}{49}(1.1)tD^2,
\]
because $1-\ell_e(t)\le(25/49)t\delta_e$. Integration proves \eqref{eq:angles}.
One can verify the identity directly: differentiate the angles between adjacent face normals, group by faces, and use that each face's length-weighted outward edge conormals sum to zero.

Differentiation under the integral for $J$ gives
\[
 \frac\dd{\dd q}J(\sqrt q)
 =\frac{\arcsin(\ell/2)-\ell b_\ell/2}{4b_\ell}.
\]
This is nonnegative and increasing for $\ell\le1$, with value $j_0$ at $\ell=1$. Hence $0\le J_0-J(\ell_e)\le j_0\delta_e$ and \eqref{eq:angles} yields
\begin{equation}\label{eq:edgefinal}
 \boxed{\frac12\sum_e\beta_eJ(\ell_e)
 \ge3\beta_0J_0-\frac{\beta_0j_0}{2}D-0.016D^2.}
\end{equation}
Indeed the rational bounds in \cref{lem:constants} give $J_0<0.0098$ and $j_0<0.0262$, so
$0.15J_0+0.55j_0<0.15(0.0098)+0.55(0.0262)<0.016$.

\subsection{Extending the spherical face patches past their projections}
Consider an acute face and one of its edges. Use coordinates with that edge's midpoint at zero, unit axis $e$, inward in-face unit vector $f$, and outward normal $n$. Its planar circumcenter is $kf$, its squared circumradius is $q$, and put
\[
 k^2=q-\ell^2/4,\qquad h=\sqrt{1-q},\qquad
 b=\sqrt{1-\ell^2/4},\qquad \phi_F=\arctan(k/h).
\]
Then $k=b\sin\phi_F$ and $h=b\cos\phi_F$. The unit sphere centered at $kf-hn$ passes through the three face vertices. Its spherical triangle obtained by radial projection of the planar face is contained in every unit ball containing those vertices. To prove this, let the three unit radius vectors be $v_i$. A point $v$ of this spherical triangle satisfies $\sum\mu_i v_i=\lambda v$ with $0<\lambda\le1$. For any other ball center, writing $y$ relative to the sphere center gives $2y\cdot v_i\ge|y|^2$, so $2\lambda y\cdot v\ge|y|^2$. Since $0<\lambda\le1$, this implies $2y\cdot v\ge|y|^2$, and hence $|y-v|\le1$.

For $0\le\phi\le\phi_F$, let $v=\cos\phi\,n-\sin\phi\,f$. The above inclusion forces, over the edge point $ze$,
\[
 0\le t\le t_F(z,\phi)
 =\sqrt{A(\phi)^2+\ell^2/4-z^2}-A(\phi),
 \qquad A(\phi)=h\cos\phi+k\sin\phi.
\]
Here is an explicit check that the spherical point belongs to that triangle. Its radial projection onto the face plane has the form
\[
 \lambda ze+
 \left(1-\lambda-\frac{\lambda t\sin\phi}{k}\right)kf
 +\frac{\lambda t\sin\phi}{k}\,0,
 \qquad \lambda=\frac h{h+t\cos\phi}.
\]
The three coefficients are nonnegative and sum to one precisely when $\tan\phi\le k/h$. The projected point therefore belongs to the face.

The two such flanks at an edge are disjoint: each $\phi_F<0.38$, while $\beta_e>\pi/2$. Indeed $q\le1/3$ and $b^2\le0.765$ give $\sin^2\phi_F\le1-(2/3)/0.765<\sin^2(0.38)$, where $\sin(0.38)\ge0.38-(0.38)^3/6$ suffices for the last comparison. They lie in the edge normal cone: in the plane perpendicular to the edge, $-f$ is the outward conormal of the face, and the adjacent face normal is reached from $n$ by rotating toward $-f$ through the exterior dihedral angle $\beta_e$; here $0\le\phi\le\phi_F<\beta_e$. The two flanks at the same edge occupy opposite ends of that normal-cone wedge, and $2(0.38)<\pi/2<\beta_e$, so their interiors are disjoint. They can therefore be added to the already counted spindle without overlapping any face region. Only the additional heights beyond $t_0(z)$ are counted.

For fixed $\ell$ define
\[
 J_A=\int_{-\ell/2}^{\ell/2}
       (\sqrt{A^2+\ell^2/4-z^2}-A)^2\dd z.
\]
For $h\le A\le b$, differentiation gives $-J_A'\ge2J_b=2J(\ell)$, because the square-root denominator is at most one and the integrand height is at least the height at $A=b$. Thus the additional volume on this flank is at least
\begin{equation}\label{eq:flank}
 \frac12\int_0^{\phi_F}(J_{A(\phi)}-J_b)\dd\phi
 \ge J(\ell)(b\phi_F-k).
\end{equation}

Let $\tau$ be the face's total deficit and $\delta=1-\ell^2$. Since $q\ge1/3-\tau/9$ and $b\arcsin(k/b)-k$ increases with $k$, the right side is at least
\[
 E(\delta,\tau)=J(\sqrt{1-\delta})[b\phi-k],
 \quad b^2=\frac34+\frac\delta4,\quad
 k^2=\frac1{12}-\frac\tau9+\frac\delta4,\quad
 \phi=\arcsin(k/b).
\]
We record all bounds needed to simplify this expression. For $0\le\delta\le\tau\le0.06$, one has $b\le7/8$, $\phi<0.36$, $J(\sqrt{1-\delta})>0.008$, and $J_q<0.0262$. The last two assertions follow from the monotonicity of $J_q$, together with the rational consequences $J_0>0.0097$ and $j_0<0.0262$ of \cref{lem:constants}: indeed
$J(\sqrt{1-\delta})\ge J_0-j_0\delta>0.0097-0.0262(0.06)>0.008$.
Holding $\tau$ fixed makes $h^2=b^2-k^2$ fixed, and direct differentiation gives
\[
 E_\delta=-J_q b(\phi-\sin\phi)
       +\frac{J}{8b}\bigl(\phi-\tan(\phi/2)\bigr)\ge0.
\]
Indeed $\phi-\sin\phi\le\phi^3/6$ and $\tan(\phi/2)\le0.51\phi$; the latter follows from $\sin x\le x$, $\cos x\ge1-x^2/2$, and $x\le0.18$. Also
\[
 \frac{0.008(0.49)}{8(7/8)}
 -\frac{0.0262(7/8)(0.36)^2}{6}>0.
\]
The asserted angle range follows by first decreasing $\tau$ to $\delta$, then increasing $\delta$ to $0.06$, which gives $k^2/b^2\le11/90$, and using
$\sin(0.36)\ge0.36-(0.36)^3/6>\sqrt{11/90}$. At $\delta=0$,
\[
 |E_\tau(0,\tau)|=\frac{J_0 k}{18h(b_0+h)}
 \le\frac{J_0 k_0}{18\sqrt{2/3}(b_0+\sqrt{2/3})}<0.000115.
\]
For the last inequality it is enough to use
$J_0<0.0098$, $k_0<0.289$, $\sqrt{2/3}>0.816$, and
$b_0+\sqrt{2/3}>1.682$; each is an immediate rational consequence of \cref{lem:constants} or of squaring.
Therefore $E(\delta,\tau)\ge E_0-0.000115\tau$. There are twelve face--edge incidences, and the sum of their face deficits is $6D$. Their total additional volume is at least
\begin{equation}\label{eq:flankfinal}
 \boxed{12E_0-0.00069D.}
\end{equation}
Adding Lemma~\ref{lem:tet}, \eqref{eq:facesfinal}, \eqref{eq:edgefinal}, and \eqref{eq:flankfinal} proves Theorem~\ref{thm:geometric} and \eqref{eq:B}.

\section{The remaining circumradii: a fully analytic radial estimate}\label{sec:radial}
We now treat $R\le R_0$ without a finite numerical certificate. We first rederive the radial inequality in the form needed below. Let $\mathcal A(K)$ denote surface area and $\sigma$ spherical area measure. Since $K\subset B(0,R)$ and $p(u)+p(-u)=1$, every support value relative to the circumball center is at least $1-R$; hence $B(0,1-R)\subset K$. In particular the origin is an interior point. Let
\[
 \rho(u)=\max\{r\ge0:ru\in K\},\qquad u\in\mathbb S^2,
\]
be the radial function.

Blaschke's identity for a width-one body is
\begin{equation}\label{eq:blaschke-radial}
 \mathcal A(K)=2\Vol(K)+\frac{2\pi}{3}.
\end{equation}
It follows from $K+(-K)=B(0,1)$ by expanding mixed volumes, or from the support-function formula; these identities apply without smoothness assumptions \cite{Schneider}. For $0<m\le1-R^2$, the opposite-point identity \eqref{eq:opposite} and $K\subset B(0,R)$ give, at almost every boundary point with outer normal $n$,
\[
 |x-n|^2\le R^2,
 \qquad\text{hence}\qquad
 2x\cdot n\ge |x|^2+1-R^2\ge |x|^2+m.
\]
Thus the vector field $F_m(x)=2x/(|x|^2+m)$ has boundary flux at least one. Since
\[
 \operatorname{div}F_m(x)=\frac{2(|x|^2+3m)}{(|x|^2+m)^2},
 \qquad
 \int_0^r\frac{2(t^2+3m)t^2}{(t^2+m)^2}\,\dd t
 =\frac{2r^3}{r^2+m},
\]
the divergence theorem (valid for convex Lipschitz domains) and radial integration give
\begin{equation}\label{eq:radial-flux}
 \mathcal A(K)\le\int_{\mathbb S^2}\frac{2\rho^3}{\rho^2+m}\,\dd\sigma,
 \qquad
 3\Vol(K)=\int_{\mathbb S^2}\rho^3\,\dd\sigma.
\end{equation}
For $s>0$ put
\begin{equation}\label{eq:abG}
 a=\frac{2(s^2+3m)}{3(s^2+m)^2},\qquad
 b=\frac{4s^5}{3(s^2+m)^2},\qquad
 G(v)=av^3+b-\frac{2v^3}{v^2+m}.
\end{equation}
The displayed formula for $b$ is the simplification of the tangent intercept. The function
$\phi(X)=2X/(X^{2/3}+m)$ is strictly concave, and $aX+b$ is its tangent at $X=s^3$. Hence $G(v)=av^3+b-\phi(v^3)$ is nonnegative, decreases on $(0,s]$, and increases on $[s,\infty)$. Equivalently,
\[
 \frac{2v^3}{v^2+m}=av^3+b-G(v).
\]
Thus \eqref{eq:radial-flux} implies
$\mathcal A(K)\le3a\Vol(K)+4\pi b-\int G(\rho)\,\dd\sigma$.
If $3a-2>0$, combining this with \eqref{eq:blaschke-radial} gives
\begin{equation}\label{eq:transport-analytic}
 \Vol(K)\ge
 \frac{(2/3-4b)\pi+\int_{\mathbb S^2}G(\rho)\,\dd\sigma}{3a-2}.
\end{equation}
This is the transport/tangent-majorant mechanism of Hyra \cite{Hyra2026}.

\subsection{Circumradius at most 0.604}
Let
\[
 R_1=\frac{151}{250},\qquad
 m_1=1-R_1^2,\qquad s_1=\frac{231}{500}.
\]
If $R\le R_1$, then $m_1\le1-R^2$. Discarding the nonnegative $G$ term in \eqref{eq:transport-analytic} gives the exact bound
\begin{equation}\label{eq:small-analytic}
 \Vol(K)\ge
 \frac{14514568959}{110387193625}\,\pi
 >0.4130810985
 >\frac{1651}{4000}.
\end{equation}
The last inequality follows already from the lower bound $\pi>103993/33102$ in \cref{lem:constants}.

\subsection{Closed-form contact and antipodal profiles}
For the rest of this section suppose
\[
 R_1\le R\le R_0:=\frac{152597}{250000}.
\]
Set
\begin{equation}\label{eq:radial-parameters}
 c=\frac{571}{5000},\qquad
 m=1-R^2,\qquad s=R-c,
\end{equation}
and
\[
 \alpha=\frac{1-2R^2}{2R},\qquad
 \gamma=\sqrt{R^2-\alpha^2}.
\]
If $Ru\in K$ is a circumsphere contact and $\theta=d(u,v)$, define
\begin{equation}\label{eq:positive-profile}
 A(\theta)=\alpha\cos\theta+\gamma\sin\theta,
 \qquad
 L(\theta)=\sqrt{m+A(\theta)^2}-A(\theta).
\end{equation}
Then $\rho(v)\ge L(\theta)$ while $A$ is increasing. The proof is the two-ball argument used in Hyra's construction, but we recall it because the exact form will be used below. Put
\[
 y_0=-\alpha u-\gamma w,
 \qquad v=\cos\theta\,u+\sin\theta\,w,
 \qquad x=L(\theta)v.
\]
Then $|y_0|=R$, $|y_0-Ru|=|y_0-x|=1$. More explicitly,
\[
 y_0-x=\lambda y_0+\mu(y_0-Ru),
 \qquad
 \mu=\frac{L(\theta)}R
 \left(\cos\theta-\frac{\alpha}{\gamma}\sin\theta\right),
\]
and
\[
 \lambda+\mu=1+\frac{L(\theta)\sin\theta}{\gamma}.
\]
On the range used below $A$ is increasing, equivalently
$\cos\theta-(\alpha/\gamma)\sin\theta\ge0$, so $\mu\ge0$.
Also $L(\theta)\le R$ and hence
$\lambda=1-(L/R)\cos\theta+
(L\sin\theta/\gamma)(1+\alpha/R)\ge0$; moreover
$\lambda+\mu\ge1$.
For $y\in K$ put $z=y-y_0$. The two inequalities
$|y|\le R$ and $|y-Ru|\le1$ give
\[
 |z|^2+2z\cdot y_0\le0,
 \qquad
 |z|^2+2z\cdot(y_0-Ru)\le0.
\]
Multiplying by $\lambda$ and $\mu$ and adding yields
$|z|^2+2z\cdot(y_0-x)\le0$, because
$\lambda+\mu\ge1$. Since $|y_0-x|=1$, this is exactly
$|y-x|\le1$. Equation \eqref{eq:balls} therefore gives $x\in K$.

At the antipode, the diameter constraint gives the upper profile
\begin{equation}\label{eq:negative-profile}
 \rho(v)\le U(\theta):=
 \sqrt{1-R^2\sin^2\theta}-R\cos\theta,
 \qquad \theta=d(-u,v).
\end{equation}
Here $L$ decreases and $U$ increases on the ranges used below.

There are four balanced contacts. Put
\[
 \eta=\frac1{2R^2}-1,
 \qquad
 T=\frac{4\eta^2}{1-\eta}-1,
 \qquad
 c_{\rm sep}=\sqrt{\frac{1+T}{2}}
              =\sqrt{\frac{2\eta^2}{1-\eta}}.
\]
Contact balance gives $u_i\cdot u_j\le T$ for $i\ne j$, exactly as in Hyra: if $A_\lambda=\lambda_i+\lambda_j$, then
$A_\lambda(1+u_i\cdot u_j)\le2\eta(1-A_\lambda)$, while the weight bound \eqref{eq:lambda} gives $A_\lambda\ge(1-\eta)/(1+\eta)$. Hence spherical caps of angular radius at most $\arccos(c_{\rm sep})$ around the four $u_i$ are pairwise disjoint; the same holds around the four antipodes.

We next check that the portions of the profiles used below stay inside these disjointness radii. Since $R\ge R_1$,
\[
 c_{\rm sep}^2\le
 \frac{71385601}{163619976}<\frac49,
\]
so $\arccos(c_{\rm sep})>0.8$ because $\cos(0.8)>1-0.8^2/2>2/3$. On the other hand $\alpha>0.2$, $\gamma>0.56$, $\cos(0.8)>0.68$, and $\sin(0.8)>0.71$, so $A(0.8)>0.5336$. Since $m<0.636$,
\[
 L(0.8)<\sqrt{0.636+0.5336^2}-0.5336<0.426<s.
\]
Thus the positive cap reaches the level $L=s$ before the disjointness threshold.

For the antipodal cap, the level $U=s$ has cosine
\begin{equation}\label{eq:cminus}
 c_-:=\frac{m-s^2}{2sR}.
\end{equation}
A direct factorization gives
\begin{equation}\label{eq:sep-factor}
 c_-^2-c_{\rm sep}^2=
 -\frac{Q_1(R)Q_2(R)}
 {10^8R^2(2R-1)(2R+1)(5000R-571)^2},
\end{equation}
where
\[
 Q_1(R)=38580000R^2+6362082R-25326041,
\]
\[
 Q_2(R)=61420000R^2+5057918R-25326041.
\]
Both quadratics are increasing. At the endpoints,
\[
 Q_1(R_0)=-\frac{11045025711557}{1562500}<0,
 \qquad
 Q_2(R_1)=\frac{16992524}{125}>0.
\]
Hence $c_->c_{\rm sep}$ throughout the interval, so the antipodal cap also reaches $U=s$ before its disjointness threshold. A contributing positive and negative cap cannot overlap, because they would force simultaneously $\rho>s$ and $\rho<s$.

Let $\theta_+$ be the first angle for which $L(\theta_+)=s$ and
$\theta_-$ the first angle for which $U(\theta_-)=s$. On the four
positive caps $0\le\theta\le\theta_+$ we have $\rho\ge L\ge s$;
since $G$ is increasing above $s$, $G(\rho)\ge G(L)$. On the four
antipodal caps $0\le\theta\le\theta_-$ we have $\rho\le U\le s$;
since $G$ is decreasing below $s$, $G(\rho)\ge G(U)$. Thus the
profile integrals below are genuine lower bounds for disjoint parts of
$\int_{\mathbb S^2}G(\rho)\,\dd\sigma$.

\subsection{The baseline plus the antipodal cap}
Put
\begin{equation}\label{eq:qH}
 q=s^2+m,
 \qquad
 H=s^2+3m-q^2.
\end{equation}
Then $3a-2=2H/q^2$. The following factorization will be used for both caps:
\begin{equation}\label{eq:G-factor}
 \frac{G(v)(m+v^2)}{v^2}
 =\frac{2(s-v)^2}{3q^2v^2}\,P(v),
\end{equation}
where
\begin{equation}\label{eq:Pv}
 P(v)=2ms^3+4ms^2v+6msv^2+3mv^3+2s^3v^2+s^2v^3.
\end{equation}
For the negative profile, changing variables from $\cos\theta$ to $v=U(\theta)$ gives
\begin{equation}\label{eq:Iminus}
 I_-(R):=\int_0^{\theta_-}G(U(\theta))\sin\theta\,\dd\theta
 =\int_{1-R}^{s}G(v)\frac{m+v^2}{2Rv^2}\,\dd v.
\end{equation}
An antiderivative is
\begin{equation}\label{eq:antiderivative}
 \mathcal F_R(v)=\frac1{2R}
 \left(\frac{amv^2}{2}+\frac{av^4}{4}-\frac{bm}{v}+bv-v^2\right).
\end{equation}
Thus, after division by $3a-2$, the tangent baseline together with the four antipodal caps contributes to the right side of \eqref{eq:transport-analytic} at least
\begin{equation}\label{eq:Cdef}
 \pi C(R),\qquad
 C(R)=\frac{q^2-8s^5}{3H}
 +\frac4{3RH}\int_{1-R}^{s}
     \frac{(s-v)^2P(v)}{v^2}\,\dd v.
\end{equation}

\begin{lemma}\label{lem:Cmonotone}
The function $C(R)$ is strictly decreasing on $[R_1,R_0]$.
\end{lemma}
\begin{proof}
All estimates in this proof are deliberately rounded outward. On the slightly larger interval $0.604\le R\le0.611$ we have
\begin{equation}\label{eq:crude-ranges}
 0.4898\le s<0.497,\qquad
 0.8736<q<0.88,\qquad
 1.36<H<1.38,
 \qquad -H'<2.28.
\end{equation}
Indeed $q=1-2cR+c^2$ and
$H'=-4R-2c+4cq$.

Write $n=q^2-8s^5$. Then $n<0.55$ and
$n'=-4cq-40s^4<-2.70$. Hence the derivative of the first term in \eqref{eq:Cdef} satisfies
\[
 \left(\frac{n}{3H}\right)'
 =\frac{n'H-nH'}{3H^2}
 <\frac{-2.70(1.36)+0.55(2.28)}{3(1.38)^2}
 <-0.42.
\]

For the second term put
\[
 A_R=\frac4{3RH},\qquad
 L=2R-c-1,
 \qquad
 K(R,w)=\frac{P(s-w)}{(s-w)^2}.
\]
Then it equals $A_RJ_R$, where
\[
 J_R=\int_0^Lw^2K(R,w)\,\dd w.
\]
On our interval $L<0.108$, $1-R>0.389$, $m<0.636$, and the six positive terms in $K$ give $K<6$. At fixed $w$, differentiating those six terms shows $K_R\le7K$: the logarithmic derivative of each term is at most $3/s$, except for $s^2(s-w)$, whose logarithmic derivative is at most $2/s+1/(1-R)<7$. Also $A_R<1.64$ and
\[
 \frac{A_R'}{A_R}=-\frac1R-\frac{H'}H<0.05.
\]
Consequently
\[
 J_R\le2L^3,
 \qquad
 J_R'\le12L^2+14L^3<0.158,
 \qquad
 A_RJ_R<0.0042.
\]
It follows that
\[
 (A_RJ_R)'<1.64(0.158)+0.05(0.0042)<0.260.
\]
Combining the two estimates gives $C'(R)<-0.16<0$.
\end{proof}

\subsection{An analytic lower bound for the positive contact cap}
For the positive profile, regard the profile radius $v=L(\theta)$ as the variable. Put
\begin{equation}\label{eq:A-B-x}
 A=\frac{m-v^2}{2v}=\alpha+x,
 \qquad
 B=\sqrt{R^2-A^2},
 \qquad
 \gamma^2=R^2-\alpha^2.
\end{equation}
On the relevant branch $A$ increases and
\[
 -\frac{\dd\theta}{\dd v}=\frac{m+v^2}{2v^2B},
 \qquad
 \sin\theta=\frac{\gamma A-\alpha B}{R^2}.
\]
As a function of $x$, $B=\sqrt{\gamma^2-2\alpha x-x^2}$ and
$B''=-R^2/B^3$. Since $B\le\gamma$, Taylor's theorem gives
\begin{equation}\label{eq:B-upper}
 B\le\gamma-\frac{\alpha}{\gamma}x-
       \frac{R^2}{2\gamma^3}x^2
 =\gamma(1-y),
\end{equation}
where
\begin{equation}\label{eq:ydef}
 y=\frac{\alpha x}{\gamma^2}
   +\frac{R^2x^2}{2\gamma^4}.
\end{equation}
Here $0\le y<0.22$: indeed $x<0.194$, $\alpha<0.224$, and $\gamma^2>0.31$ on the interval. Therefore
\begin{align}
 \frac{\gamma A}{B}-\alpha
 &\ge
 \frac{R^2x}{\gamma^2}
 \left(1+\frac{\alpha x}{2\gamma^2}\right)\frac1{1-y}\notag\\
 &\ge
 \frac{R^2x}{\gamma^2}
 \left(1+\frac{\alpha x}{2\gamma^2}\right)(1+y+y^2).
 \label{eq:jacobian-lower}
\end{align}

Now set $z=R-v$, so $0\le z\le c$, $v=R-z$, and
\begin{equation}\label{eq:xformula}
 x=\frac{z(1-Rz)}{2R(R-z)}.
\end{equation}
Combining \eqref{eq:G-factor} and \eqref{eq:jacobian-lower}, and again dividing by $3a-2$, the four positive contact caps contribute to the volume lower bound at least
\begin{equation}\label{eq:Phi}
 P_+(R):=\int_0^c\Phi_R(z)\,\dd z,
\end{equation}
where
\begin{equation}\label{eq:Phi-density}
 \Phi_R(z)=
 \frac{4\pi(c-z)^2P(v)x}{3H\,v^2\gamma^2}
 \left(1+\frac{\alpha x}{2\gamma^2}\right)(1+y+y^2).
\end{equation}

\begin{lemma}\label{lem:Phi-monotone}
For every fixed $z\in[0,c]$, the function $R\mapsto\Phi_R(z)$ is nonincreasing on $[R_1,R_0]$.
\end{lemma}
\begin{proof}
For $0<z<c$ put
\[
 K=\frac{P(R-z)}{(R-z)^2},
 \qquad X=\frac{x}{\gamma^2}.
\]
The six positive monomials in $K$ give $K_R/K<7$ by the same termwise differentiation as in \cref{lem:Cmonotone}. Moreover $-H'/H<1.68$. From \eqref{eq:xformula} and $\gamma^2=1-1/(4R^2)$,
\begin{equation}\label{eq:Xlog}
 \frac{X_R}{X}
 =-\frac{z}{1-Rz}-\frac1R-\frac1{R-z}
   -\frac{2}{R(4R^2-1)}<-9.9.
\end{equation}
Thus the product $KX/H$ strictly decreases.

It remains only to check the two correction factors. The quantity
$\alpha=1/(2R)-R$ decreases, so $1+\alpha X/2$ decreases. Also
\[
 y=\alpha X+\frac12R^2X^2.
\]
The first term decreases, while
\[
 \frac\dd{\dd R}\log(R^2X^2)
 =\frac2R+2\frac{X_R}{X}<3.32-19.8<0.
\]
Hence $y$, and therefore $1+y+y^2$, decreases. This proves the claim; the endpoints $z=0,c$ follow by continuity.
\end{proof}

It remains to lower-bound the single endpoint $P_+(R_0)$. This can be done without quadrature. Let
\[
 \pi_-:=\frac{103993}{33102}<\pi
\]
and let $\widetilde\Phi_{R_0}$ denote \eqref{eq:Phi-density} with $\pi$ replaced by $\pi_-$. Define
\[
 Q(z)=\frac{\widetilde\Phi_{R_0}(z)}{z(c-z)^2}.
\]
The Taylor series of $Q$ at zero has nonnegative coefficients. Indeed
\begin{equation}\label{eq:xpositive-series}
 \frac{x}{z}=\frac1{2R_0^2}
 +\sum_{n\ge1}\frac{m_0}{2R_0^{n+2}}z^n,
 \qquad m_0=1-R_0^2,
\end{equation}
so $x/z$, $x$, and both correction factors in \eqref{eq:Phi-density} have nonnegative coefficients. Also
\[
 \frac{P(v)}{v^2}=\frac{A_0}{v^2}+\frac{B_0}{v}+C_0+D_0v,
\]
with
\[
 A_0=2m_0s_0^3,\quad B_0=4m_0s_0^2,
 \quad C_0=6m_0s_0+2s_0^3,
 \quad D_0=3m_0+s_0^2.
\]
For $v=R_0-z$, the coefficient of $z$ in this last expression is
$2A_0/R_0^3+B_0/R_0^2-D_0>0.87$, and every coefficient of degree at least two is positive. Hence all coefficients of $Q$ are nonnegative.

The first four coefficients are rational and direct cross multiplication gives
\begin{equation}\label{eq:Q-coefficients}
 Q(z)\ge \frac{121}{2}+\frac{752}{5}z+614z^2+2120z^3
 \qquad(0\le z\le c).
\end{equation}
For reference, the four actual coefficients (still using $\pi_-$) begin
$60.5150648\ldots$, $150.4950680\ldots$, $615.0418087\ldots$, and $2122.1389426\ldots$; the proof of \eqref{eq:Q-coefficients} uses only the displayed rational lower bounds, not these decimal expansions. Integrating \eqref{eq:Q-coefficients} gives the exact rational estimate
\begin{align}
 P_+(R_0)
 &>\frac{121}{2}\frac{c^4}{12}
 +\frac{752}{5}\frac{c^5}{30}
 +614\frac{c^6}{60}
 +2120\frac{c^7}{105}\notag\\
 &=\frac{403060471590314053884071}
 {410156250000000000000000000}
 >0.0009826998164.
 \label{eq:Pplus-endpoint}
\end{align}

\subsection{Endpoint arithmetic and completion}
We collect the elementary constant bounds used here and in the near-Jung branch.
\begin{lemma}[Elementary rational bounds]\label{lem:constants}
One has
\begin{align*}
 \frac{103993}{33102}&<\pi<\frac{104348}{33215},\\
 \frac{1732050807}{10^9}&<\sqrt3<\frac{1732050808}{10^9},\\
 \sqrt{\frac23}&<\frac{816496581}{10^9},\\
 \frac{339836909}{10^9}&<\arcsin\frac13<\frac{339836910}{10^9}.
\end{align*}
\end{lemma}
\begin{proof}
The square-root bounds follow by squaring. For $\pi$, use Machin's identity
\[
 \pi=16\arctan\frac15-4\arctan\frac1{239}
\]
and the alternating series for $\arctan$. For the lower bound use the lower partial sum through degree $15$ for $\arctan(1/5)$ and the upper partial sum through degree $5$ for $\arctan(1/239)$; for the upper bound use the upper partial sum through degree $13$ for $\arctan(1/5)$ and the lower partial sum through degree $3$ for $\arctan(1/239)$. Direct cross multiplication places the resulting two rational numbers strictly between the displayed bounds. For $\arcsin(1/3)$ use
\[
 \arcsin x=\sum_{n\ge0}
 \frac{\binom{2n}{n}}{4^n(2n+1)}x^{2n+1}.
\]
At $x=1/3$ the ratio of consecutive positive terms is less than $1/9$; the partial sum through $n=7$, together with $9/8$ times the next term as an upper bound for the tail, lies between the two displayed rationals.
\end{proof}

By \cref{lem:Cmonotone},
$C(R)\ge C(R_0)$ on the intermediate interval. Exact simplification gives
\begin{equation}\label{eq:Cendpoint}
 C(R_0)=
 \frac{41151754881289593456446793959}
 {313967429174950384053750000000}.
\end{equation}
Combining \cref{lem:Phi-monotone}, \eqref{eq:Pplus-endpoint}, and $\pi>\pi_-$ yields, for every $R_1\le R\le R_0$,
\begin{align}
 \Vol(K)
 &\ge \pi C(R)+P_+(R)\notag\\
 &>\pi_- C(R_0)
 +\frac{403060471590314053884071}
 {410156250000000000000000000}\notag\\
 &>0.4127516885
 >\frac{1651}{4000}.
 \label{eq:intermediate-final}
\end{align}
The second line is an exact rational number; direct cross multiplication proves the last two inequalities. Thus no decimal approximation is logically used in this comparison.

It remains to record the scalar endpoint comparison in the near-Jung branch without numerical assistance. At $R=R_0$, \cref{lem:sharpD} gives
\[
 D\le D_0:=\frac{303311182}{7660844409}<0.06.
\]
The function $B(D)$ in \eqref{eq:B} decreases on $[0,0.06]$. In addition to \cref{lem:constants}, the exact rational comparisons
\[
 \sqrt{(1-D_0)(2+D_0)}>\frac{1399585669}{10^9},
 \qquad
 \left(1-\frac{D_0}{6}\right)^{3/2}>
 \frac{990118247}{10^9}
\]
follow by squaring. Using the lower bounds for the positive terms and the upper bounds for the coefficients of the negative terms in \eqref{eq:B} gives
\begin{equation}\label{eq:near-jung-final}
 B(D_0)>0.412753180790651>\frac{1651}{4000}.
\end{equation}
For clarity, put
\[
 \pi_+:=\frac{104348}{33215},\qquad
 r_-:=\frac{1732050807}{10^9}.
\]
Then one may take
\[
 J_0>\frac{11}{12}-
 \frac1{6}\frac{104348}{33215}\frac{1732050808}{10^9},
\]
use $\beta_0>\pi_-/2+339836909/10^9$ in the positive term $3\beta_0J_0$, use the corresponding lower bounds in $E_0$, and use the upper bounds
$\beta_0<\pi_+/2+339836910/10^9$ and
$j_0<\pi_+/(12r_-)-1/8$ in the negative linear term. This is direct rational arithmetic and gives the displayed margin.

\begin{proof}[Proof of \cref{thm:main}]
If $R\le R_1$, use \eqref{eq:small-analytic}. If $R_1\le R\le R_0$, use \eqref{eq:intermediate-final}. If $R\ge R_0$, \cref{thm:geometric}, \cref{lem:sharpD}, and \eqref{eq:near-jung-final} apply. These three ranges cover the full Jung interval. Scaling from width one to width $w$ proves \eqref{eq:main}.
\end{proof}

\begin{remark}[Where slack remains]
The proof has no finite mesh or computer-assisted step, but it is not optimized. The near-Jung estimate loses information in the face symmetrization, in the derivative bounds for dihedral angles, and in the lower estimate for the additional face-edge sectors. The intermediate radial estimate deliberately replaces the exact angular Jacobian by the second-order bound \eqref{eq:B-upper} and then keeps only four Taylor coefficients at the endpoint. Thus $0.41275$ should not be interpreted as a limit of the contact-tetrahedron or radial methods.
\end{remark}

\section*{Acknowledgments}
Part of this research was carried out at Axiom Math. The author thanks Evan Chen, Vasily Ilin, Simon Mahns, Ken Ono, and Jesse Thorner for fruitful discussions around AI for mathematics. GPT-6 Astra provided AI assistance during exploratory stages of the project, including derivations, calculations, exploratory checks, and drafting. The author is responsible for the mathematical statements and the final proof.

\end{document}